\documentclass[a4paper, 12pt]{article}
\usepackage[cp1251]{inputenc}
\usepackage[english]{babel}
\usepackage{amsthm}
\usepackage{cite}
\usepackage{amsmath}
\usepackage{amsfonts}
\usepackage{amssymb}
\usepackage{amsrefs} 

\usepackage[]{authblk}

\usepackage[shortlabels]{enumitem}

\DeclareMathOperator{\diag}{diag}
\DeclareMathOperator{\Spec}{Spec}
\DeclareMathOperator{\Frac}{Frac}
\DeclareMathOperator{\Trdeg}{Trdeg}
\DeclareMathOperator{\tr}{tr}
\DeclareMathOperator{\Ad}{Ad}
\DeclareMathOperator{\Ch}{char}
\DeclareMathOperator{\PI}{PI}

\makeatletter
\newsavebox{\@brx}
\newcommand{\llangle}[1][]{\savebox{\@brx}{\(\m@th{#1\langle}\)}
  \mathopen{\copy\@brx\kern-0.5\wd\@brx\usebox{\@brx}}}
\newcommand{\rrangle}[1][]{\savebox{\@brx}{\(\m@th{#1\rangle}\)}
  \mathclose{\copy\@brx\kern-0.5\wd\@brx\usebox{\@brx}}}
\makeatother

\newtheorem{thm}{Theorem}[section]
\newtheorem{lem}[thm]{Lemma}

\newtheorem{prop}[thm]{Proposition}
\newtheorem{coro}[thm]{Corollary}

\newtheorem{defn}[thm]{Definition}
\newtheorem{remark}[thm]{Remark}
\newtheorem{problem}[thm]{Problem}
\begin{document}
\renewcommand{\thefootnote}{\fnsymbol{footnote}}

\renewcommand{\thefootnote}{\arabic{footnote}}
\fontsize{12}{12pt}\selectfont
\title{\bf Centralizers in Free Associative Algebras and Generic Matrices}
\renewcommand\Affilfont{\itshape\small}
\author[2,3]{Alexei Belov-Kanel\thanks{kanel@mccme.ru}}
\author[3,4]{Farrokh Razavinia\thanks{f.razavinia@phystech.edu}}
\author[1]{Wenchao Zhang\thanks{Correspondence: zhangwc@hzu.edu.cn}}

\affil[1]{School of Mathematics and Statistics, Huizhou University, Huizhou, 516007, China}
\affil[2]{Department of Mathematics, Bar-Ilan University, Ramat Gan, 5290002, Israel}
\affil[3]{Department
of Discrete Mathematics, Moscow Institute of Physics and Technology, Dolgoprudnyi, Institutskiy Pereulok,
141700 Moscow, Russia}
\affil[4]{Department of Mathematics, Azerbaijan Shahid Madani University, East Azerbaijan Province, Tabriz, Iran}

\date{}

\maketitle

\renewcommand{\abstractname}{Abstract}
\begin{abstract}
This paper is concerned with the completion of the proof of  the Bergman centralizer theorem by using generic matrices based on our previous quantization proof \cite{KBRZh}. Additionally, we establish that the algebra of generic matrices with characteristic coefficients is integrally closed.

\end{abstract}

\section*{Introduction}

Let $K$ denote a base field and $A$ be an (associative) $K$-algebra. For an element $a\in A$, we denote by $C(a;A)$ the set of all elements of $A$ that commute with $a$. We say that $C(a;A)$ is the \textit{centralizer} of $a$ in $A$. The study of centralizers plays an important role in investigating the structure of algebras and there are numerous results in the literature relating to centralizers. The Bergman centralizer theorem \cite{Berg}, which states that the centralizer of any non-constant element in the free associative algebra is a polynomial algebra on a single variable, is one of the most significant and essential conclusions. This theorem plays a crucial role in studying algorithmic and combinatorial questions.

\medskip

Nonetheless, as far as we know, no (conceptual) new proof has been presented for this theorem since G. Bergman \cite{Berg} nearly fifty years ago. This has inspired us to take a fresh look at this result. We use a method of deformation quantization motivated by M. Kontsevich \cite{Kon03} to obtain an alternative proof of the Bergman centralizer theorem. In our previous work \cites{KBRZh}, we demonstrated that the centralizer of a non-scalar element in $K\langle X\rangle$ is a commutative subalgebra of transcendence degree one. A detailed proof of above result can be found in Zhang's thesis \cite{Zhang}.

For us, it was the most interesting part of the proof of the Bergman centralizer theorem. However, it is reasonable to provide the entire proof. To complete the argument, we invoke the following result due to P.M. Cohn \cite{Cohn}*{Prop. 2.1} which uses the Luroth theorem, i.e., any subalgebra of the algebra $K[x]$ is free if and only if it is integrally closed. Therefore, in this work, we will concentrate more on the demonstration of the integral closedness of centralizers. Our method, which differs from the conventional one in some ways, makes use of the $PI$-theory. From a classical perspective, it is more or less simple to demonstrate that the centralizer is of transcendence degree 1, but from the perspective of our approach, it is simpler to demonstrate that the centralizer is integrally closed. The proof of the integral closedness of the centralizer in this work is much concise than Bergman's initial proof (it uses invariant theory instead). For our approach, we still require a statement from Bergman \cite{Berg} (i.e., Proposition \ref{propberg}).

\medskip

Despite the profound noncommutative divisibility theorems of Cohn and Bergman, we adopt a characteristic-free method in our demonstration, which entails the use of generic matrices reduction and invariant theory for characteristic zero by C. Procesi \cite{Pro76} and for positive characteristic due to A. N. Zubkov \cites{Zubkov, Zubkov2} and S. Donkin \cites{Donkin, Donkin2}.

\medskip

We first think about the localization of the algebra of generic matrices by transferring centralizers of non-scalar elements from the free associative algebra case into the algebra of generic matrices.

\begin{thm}     \label{ThmGenMatrixdomain}
The algebra of generic matrices is a domain. Its localization as a skew field coincides with the localization of the algebra of generic matrices with traces (for positive characteristic -- with forms).
\end{thm}

The algebra of generic matrices is subsequently shown to be integrally closed.

\begin{thm}     \label{ThMain2}
  The algebra of generic matrices with characteristic coefficients is integrally closed.
\end{thm}

\medskip

Recently, N. Miasnikov considered an analog of the Bergman centralizer theorem for free group algebras in his really fascinating study \cite{Mia16}. He established that the centralizer of any non-scalar element of a free group algebra over a field is the coordinate ring of a nonsingular curve.

The paper is organized as follows. Theorem \ref{ThmGenMatrixdomain} is proved in Section \ref{GenMatrix} after a review of some fundamental results on generic matrices. The proof of our main Theorem \ref{ThMain2} is covered in Section \ref{cent-int} together with the integrally closedness of the algebra of generic matrices as well as centralizers of the free associative algebra. In Section \ref{final-proof}, we complete our revised proof of the Bergman centralizer theorem, and in Section \ref{Discourtions}, we discuss the rationality of degree one subfields in the fraction field of generic matrices.

\subsection*{Acknowledgments}
This work is a continuation of our work which has been carried out in \cite{KBRZh} by A. Belov-Kanel, F. Razavinia and W. Zhang. These authors contributed equally to this work. We thank U. Vishne and L. Rowen for useful discussions and we sincerely thank A. Elishev for revising our writing and providing rich discussions. We also thank the anonymous referee for valuable comments improving our manuscript.

\medskip

Alexei Belov-Kanel was supported by the Russian Science Foundation (Grant No. 22-11-00177). Wenchao Zhang was supported by the GuangDong Basic and Applied Basic Research Foundation (Grant No. 2022A1515110634), the Guangdong Provincial Department of Education (Grant No. 2021ZDJS080), and the Professorial and Doctoral  Scientific Research Foundation of Huizhou University (Grant No. 2021JB022). Farrokh Razavinia was partially supported and funded by the science foundation of Azerbaijan Shahid Madani University.

\section{Algebra of generic matrices}           \label{GenMatrix}

A \textit{generic matrix} is a matrix whose entries are distinct commutative indeterminates. Let $n$ be a positive integer, $s\geq 2$ be an integer and let $X_1,\dots, X_s$ be $n\times n$ matrices with entries $x^{(\nu)}_{ij}$ which are independent central variables. Then the $K$-subalgebra of $M_n(K[x_{ij}^{(\nu)}])$ generated by the matrices $X_{\nu}$ is called the \textit{algebra of generic matrices} and will be denoted by $K\langle X_1, \dots, X_s\rangle$ or $K\{X\}$ for short. The algebra of generic matrices is a basic object in the study of polynomial identities and invariants of $n\times n$ matrices and as we have seen already in \cite{KBRZh}, there is a canonical homomorphism $\pi$ from the free associative algebra to the algebra of generic matrices: $$\pi: K\langle X\rangle\to K\{X\}:=K\langle X_1,\dots, X_s\rangle,$$

\hspace{-0.6cm}and as an important property of $\pi$ which we will use frequently, is as follows:

An element $f$ in the free associative algebra is in the kernel of the map $\pi$, if and only if $f$ vanishes identically on $M_n(R)$ for every commutative $K$-algebra $R$, and this is true if and only if it vanishes identically on $M_n(K)$.

\medskip

Note that, the kernel of the map $\pi$ is a \textit{$T$-ideal}. A $T$-ideal is a completely characteristic ideal, i.e. stable under any endomorphism. Ascending chain conditions of $T$-ideals and their representability are discussed in \cite{Belovshpecht}, and the reference therein.

\medskip

In other words, let $h$ be an endomorphism of the free associative algebra $K\langle X\rangle$ (over $s$ generators), and let $I_s$ be the $T$-ideal of identities for the algebra of generic matrices of order $s$. Then we have $h(I_s)\subseteq I_s$ for all $s$ by the definition of $I_s$. Hence $h$ induces an endomorphism $h_{I_s}$ of
$K\langle X\rangle$ modulo $I_s$. If $h$ is invertible, then $h_{I_s}$ is invertible, but the converse is not true. This fact is proven in detail by \cite{BBRJ}, and regarding the $T$-ideals, please refer to \cite{EaKaFaW}*{Definition 5.3} and the computations afterward.

\medskip

Let us denote by $\overline{f}$ the image of the homomorphism $\pi$ of the nonscalar element $f \in K\langle X\rangle$ and by $\Frac(K\langle X\rangle)$ the skew field of fractions of the free associative algebra $K\langle X\rangle$. In our previous paper \cite{KBRZh}, we already proved that the centralizer $$C:= C(f;K\langle X\rangle)$$ of $f \in K\langle X\rangle\setminus K$ is a commutative $K$-subalgebra with transcendence degree one.

\medskip

We focus on the algebra of generic matrices of order $n=p$, where $p$ is a big enough prime number, throughout the entire paper. In our previous work \cite{KBRZh}*{Lemma 2.3}, we showed that if $n$ is big enough, $C$ can be embedded into the algebra of generic matrices of order $n$.

\medskip

When dealing with matrices in characteristic 0, it is useful to think that they form an algebra with a further unary operation $x\mapsto \tr(x)$, which is called the
$\textit{trace}$. This can be formalized as follows \cite{BPS1}:

\begin{defn}
An algebra with trace is an algebra equipped with an additional trace structure, that is a linear map $\tr: R\to R$ satisfying the following properties
$$\tr(ab) = \tr(ba),\quad a\tr(b) = \tr(b)a,\quad \tr(\tr(a)b) = \tr(a)\tr(b) \quad \text{for all}\; a,b\in R.$$
\end{defn}

\medskip

Our main motivation for studying traces, lies in our interest in the trace rings of generic matrices. Let $M_n$ be the variety of $n\times n$-matrices $X_1,\ldots,X_s$ defined at the beginning of this section. $(M_n)^m$ will be the $m$-fold product (over $\Spec K$) $M_n\times\ldots\times M_n$. Let $G=SL_n$. Then we define
$$K_T\{X\}:=T_{m,n}=\{f:(M_n)^m\to M_n\mid f~\text{polynomial~and} ~G-\text{equivariant}\}.$$
$T_{m,n}$ is a non-commutative ring (using the multiplication in $M_n$), and its center is given by 
$$Z_{m,n}=\{f:(M_n)^m\to \Spec K\mid f~\text{polynomial~and} ~G-\text{equivariant}\}.$$
In this case, $Z_{m,n}$ is the commutative and $T_{m,n}$ is the non-commutative trace ring of $m$ generic $n\times n$-matrices.  They were first extensively studied by M. Artin and C. Procesi. Artin and Schelter proved that the maximal ideals of $Z_{m,n}$ parameterize semisimple representations of dimension $n$ of the free algebra $\left\langle X_1,\ldots,X_s \right\rangle $, and the two-sided maximal ideals of $K_T\{X\}$ correspond to the simple components of such representations \cites{Artin69,Artin81}.

\medskip

As stated earlier, let's now demonstrate Theorem \ref{ThmGenMatrixdomain} for the algebra of generic matrices.

\begin{proof}[Proof of Theorem \ref{ThmGenMatrixdomain}]
The fact that the algebra $K\{X\}$ of generic matrices is a domain and its localization is a skew field was established by Amitsur \cite{Artin}. The only thing left to do is to demonstrate that the localization of $K\{X\}$ agrees with the localization of $K_T\{X\}$, the algebra of generic matrices with traces (for positive characteristic -- with forms). 

Consider the Hamilton-Cayley identity for $n\times n$ matrices

\begin{equation} \label{EqHamCelley}
  A^n-\xi_1(A)A^{n-1}+\cdots+(-1)^n\xi_n(A)=0,
\end{equation}

for $\xi_k(A)$ sum of main minors of $A$ of order $k$. 

Let $F(z,y_1,\dots,y_{n-1};\vec{t})$ be a polynomial, that is multi-linear and skew-symmetric with respect to 
$z,y_1,\dots,y_{n-1}$ and somehow dependant to variable $\vec{t}$. 

Let $$z=x^n,\ Y_k=\{y_1,\dots,y_{n-1}\}=\{x^{n-1},\dots,1\}\backslash\{x^{n-k}\}, \ n-1\ge k\ge 0.$$

Note well that if we replace $z$ with $x^n+(-1)^k\xi_k(x)x^{n-k}$ in $F(z,Y_k;\vec{t})$ we get zero due to skew symmetry and equation (\ref{EqHamCelley}), i.e.,
$$F\left(x^n+(-1)^k\xi_k(x)x^{n-k},Y_k;\vec{t}\right)=0.$$

On the other hand, since $\{x^{n-k}\}\cup Y_k=\{x^{n-1},\dots,x^0=1\}$ for the generic matrix $x$ are linear independent,  we have $F(x^{n-1},\dots,1;\vec{t})\ne 0$ for some $F$. We can put 
$$F(y_1,\dots,y_{n-1})=\sum_{\sigma\in S_n}(-1)^{|\sigma|} t_0y_{\sigma(1)}t_1y_{\sigma(2)}\dots t_{n-1}y_{\sigma(n)}t_n,$$ 

which means that 

$$\xi_k(x)=(-1)^{k+1}\frac{F(x^n,Y_k;\vec{t})}{F(x^{n-k},Y_k;\vec{t})},$$
and $F$ can be chosen such that the denominator $F(x^{n-k},Y_k;\vec{t})$ stays non-zero.  

And by doing this, we were able to express the characteristic coefficient as a fraction of ordinary polynomials, concluding the proof of Theorem \ref{ThmGenMatrixdomain}.
\end{proof}

\begin{remark}
\begin{enumerate}
    \item The ring of forms can be contained in the localization of central polynomials by selecting $F$ as a central polynomial in the proof of Theorem \ref{ThmGenMatrixdomain}. To accomplish this, we can consider $G=h_1 F h_2$ and replace the variable $G$ with a multi-linear central polynomial.
    \item It is simpler to argue with representation forms $\xi_k$ using trace polynomials and the Razmyslov identity if $\Ch(K)=0$:
  
  \begin{align*}
  & n\cdot \tr(A)\cdot \sum_{\sigma\in S_{n^2}}(-1)^{|\sigma|}
  x_{\sigma(1)}y_1x_{\sigma(2)}y_2x_{\sigma(3)}\ldots
  x_{\sigma(n^2-1)}y_{n^2-1}x_{\sigma(n^2)}\\&= \sum_{\sigma\in S_{n^2}}(-1)^{|\sigma|}
  Ax_{\sigma(1)}y_1x_{\sigma(2)}y_2\ldots y_{n^2-1}x_{\sigma(n^2)}
  \\&+\sum_{\sigma\in S_{n^2}}(-1)^{|\sigma|}
  x_{\sigma(1)}y_1Ax_{\sigma(2)}y_2\ldots
  y_{n^2-1}x_{\sigma(n^2)}\\& +\ldots \\&+\sum_{\sigma\in S_{n^2}}(-1)^{|\sigma|}
  x_{\sigma(1)}y_1x_{\sigma(2)}y_2\ldots y_{n^2-1}Ax_{\sigma(n^2)}.
  \end{align*}

\end{enumerate}
    
\end{remark}

The discovery attributed to Amitsur \cite{Artin}*{Theorem V.10.4} (or \cite{Zhang}*{Theorem 3.2}), which shows that the algebra of generic matrices is a domain and hence it has no zero divisors, is the reason why we select a prime number rather than an arbitrarily big enough integer.

\medskip

Then we have the following proposition:

\begin{prop}\label{propCH}
Let $K_T\{X\}$ be the algebra of generic matrices with traces (forms if the characteristic is positive). Let $R$ be its center and $\mathfrak{R}$ be the algebraic closure of the field of fractions of $R$.
\begin{enumerate}
\item[a)] Every minimal polynomial of $A\in K_T\{X\}$ is irreducible over $R$ in any characteristic. In particular, $A$ is diagonalizable over $\mathfrak{R}$.
\item[b)]  All eigenvalues $\lambda_i$ of $A\in K_T\{X\}$ are roots of the minimal polynomial of $A$, with correct multiplicities.
\item[c)] The characteristic polynomial of $A$ is a power of the minimal polynomial of $A$. 
\end{enumerate}
\end{prop}

\begin{proof}
{\bf a)} Otherwise $K_T\{X\}$ has zero divisors.

{\bf b)} By Cayley--Hamilton theorem, all zeros of the characteristic polynomial of $A$ are eigenvalues, then they are zeroes of the minimal polynomial $P(x)$ of $A$ which is minimal over $R$. Hence every irreducible component over $R$ of $Q(x)$ coincides with $P(x)$, and together with Theorem \ref{ThmGenMatrixdomain}, the statements {\bf b)} and {\bf c)} follow.
\end{proof}

A significant open problem that is well-known in the community can be stated as follows:

\begin{problem}
Whether for big enough $n$, every non-scalar element in the algebra of generic matrices has a minimal polynomial that always coincide with its characteristic polynomial.
\end{problem}

This is an important open direction, on which for small $n$, the Galois group of extension quotient field of the center of
algebra of generic matrices by eigenvalues of a non-scalar element of this algebra might not be the full symmetry group.
However, it still is unknown for big enough $n$. In this regard, some related problems have been discoursed in \cite{BRMY}.

\medskip

From Proposition \ref{propCH} (c), for $n=p$,  a large enough prime, we can obtain the following corollary.

\begin{coro}\label{mincha}
Let $K\{X\}$ be the algebra of generic matrices of a large enough prime order $n:=p$. Assume $A$ is a non-scalar element in $K\{X\}$, then the minimal polynomial of $A$ coincides with its characteristic polynomial.

\end{coro}

\begin{proof}
Let $m(A)$ and $c(A)$ be the minimal and the characteristic polynomial of $A$, respectively. Note that
$\deg{c(A)}=n$, and $c(A)=(m(A))^{(k)}$. Because $A$ is not scalar, we have $\deg{m(A)}>1$ and since $n$ is prime and $k$
divides $n$, hence we have $k=1$.
\end{proof}

The following proposition holds:

\begin{prop}\label{ProPcomm}
Every generic matrix $B$ with same eigenvectors (defined over the tensor product of algebraic closure of center) as a diagonalizable matrix $A$ commutes with $A$.
\end{prop}

\begin{proof}
Consider $A, B\in M_{n\times n}(K)$. If $B$ has the same set of eigenvectors as a diagonalizable matrix, then $B$ has $n$ linearly independent eigenvectors. Since $A$ and $B$ are $n\times n$ matrices with $n$ eigenvectors, they both are diagonalizable and therefore $A=Q^{-1}D_AQ$ and $B=P^{-1}D_BP,$ where $Q$ and $P$ are matrices whose columns are eigenvectors of $A$ and $B$ associated with the eigenvalues listed in the diagonal matrices $D_A$ and $D_B$ respectively. But, according to the hypothesis, $A$ and $B$ have the same eigenvectors and this immediately will imply that $P=Q=:S$. Therefore $A=S^{-1}D_AS$ and $B=S^{-1}D_BS$, so $AB=S^{-1}D_ASS^{-1}D_BS = S^{-1}D_AD_BS$ and in a same way we have $BA= S^{-1} D_BD_AS$ and since $D_A$ and $D_B$ are diagonal matrices that commute, hence so do $A$ and $B$.
\end{proof}

\begin{prop}\label{propeigen}
Let $n$ be a prime number and $A$ be a non-scalar element of the algebra of generic matrices, then all eigenvalues of $A$ are pairwise different.
\end{prop}

\begin{proof}
The result will directly follow from Proposition \ref{propCH} and Corollary \ref{mincha}.
\end{proof}

Proposition \ref{propeigen} implies the following results.

\begin{coro}
The set of all generic matrices commuting with $A$ consists of matrices which are all diagonalizable with $A$ simultaneously in the same eigenvector basis as in $A$.
\end{coro}

If $A$ is a non-scalar matrix, then we have the following result:

\begin{coro}
$A$ is a non-scalar element of the algebra of generic matrices $K\{X\}$, then every eigenvalue of $A$ is transcendental over $K$.
\end{coro}
\begin{proof}
We may assume that the ground field is algebraically closed. If there some eigenvalues are algebraic over $K$, then the minimal polynomial for $A$ will be decomposable, which contradicts Proposition \ref {propCH}.
\end{proof}

Consider a set $X$ of generators $a_1,\dots,a_s$ homogeneous of degree 1 of algebra of generic matrices. By {\em degree} $\deg(X)$ of an element of $X$ generic matrices we mean the degree of its highest component.  The following lemma is required.

\begin{lem}
    Let $A$ and $B$ be generic matrices and $P(B)=A$ for some polynomial $P$. Then $\deg(B)\le \deg(A)$. Moreover,  $ \deg(A) =\deg(P)\cdot \deg(B)$.
\end{lem} 

\begin{proof}
    Let $n=\deg(P)$. The highest term $A^n$ is  an $n$-th power of higest term of $B$, is non zero and has degree $n\cdot\deg(A)$.
\end{proof}

We require another crucial theorem by Amitsur and Levitzki \cite{Amit}.

\begin{thm}[Amitsur-Levitzki Theorem] Matrix rings of order $n$ are polynomial identity rings in a case that the smallest identity they satisfy has a degree exactly equal to $2n$.
\end{thm}

\section{Centralizers are integrally closed}\label{cent-int}
We attempted to provide a succinct definition of $C(a;A)$, the centralizer of $a$ in $A$, in the introduction. Let's begin this section by giving the topic a rather in-depth definition. For a ring $R$ and a subset $X\subseteq R$, we denote by $C(X;R)$ the set of all elements of $R$ which commute with every element of $X$. We say that $C(X;R)$ is the centralizer of $X$ in $R$, i.e., 
$$C(X;R)=\{r\in R: rx=xr, \forall x\in X\}.$$
If $X=\{a\}$, then we simply write $C(a;R)$ instead of $C(\{a\};R)$. Clearly $C(X;R)$ is a subring of R and it contains the center $Z(R)$. It is also clear that $C(X;R)=R$ if and only if $X\subseteq Z(R)$. We are only interested in $C(a;R)$ where $a\notin Z(R)$.
\subsection{Invariant theory of generic matrices}

In this subsection, We review some essential facts from the invariant theory of generic matrices.

\medskip

Consider the algebra $K\{X\}$ of $s$-generated generic matrices of order $n$ over the ground field $K$. Let $a_{\ell}=(a_{ij}^{\ell}), 1\leq i,j\leq n, 1\leq \ell \leq s$ be its generators. Let $R=K[a_{ij}^{\ell}]$ be the ring of entries coefficients. Consider an action of matrices $M_n(K)$ on matrices in $R$ by conjugation, namely
$$\varphi_M: B\mapsto M B M^{-1}.$$ It is well-known (cf. \cites{Pro76, Pro73, Zubkov2}) that the
invariant function on this matrix can be expressed as a polynomial over traces $\tr(a_{i1}\dots a_{is})$. Any
invariant on $K\{X\}$ is a polynomial of $\tr(a_{i1}\dots a_{in})$. Note that the conjugation on $M$
induces an automorphism $\varphi_M$ of the ring $R$. Namely, $M(a_{ij})^{\ell}M^{-1}=(a'_{ij}{}^{\ell})$,
and $\varphi_M(B)$ of $R$ induces automorphism on $M_n(R)$. And for any $x\in K\{X\}$, we have
$$\varphi_M(x)=M x M^{-1}=\Ad_M(x).$$ 
Consider $x=\Ad_M^{-1}\varphi_M(x)$. Then any element of the algebra of generic matrices is invariant under $\varphi_M$.

\medskip

There is a well-known fact that we can formulate as follows:

\begin{thm}
The algebra of generic matrices with trace is an algebra of concomitants, i.e. subalgebra of $M_n(R)$ which is invariant under the action $\varphi_M$.
\end{thm}

This theorem was first proved by C. Procesi \cite{Pro76} for the ground field $K$ of characteristic zero. If $K$
is a field of positive characteristics, we have to use not only traces but also characteristic polynomials and their
linearization (cf. \cites{Donkin,Donkin2}). Relations between these invariants were discovered by C. Procesi
\cites{Pro76, Pro73} for characteristic zero and A. N. Zubkov \cites{Zubkov, Zubkov2} for characteristic $p$.
C. de Concini and C. Procesi also generalized the characteristic-free approach to invariant theory \cite{ConcProc}. For further information and background on this topic, we recommend referencing \cite{ConcProc1}.

\medskip

Now let us denote by $K_T\{X\}$ the algebra of generic matrices with traces (characteristic coefficients of {\it forms} for
positive characteristics).

\begin{prop}\label{ICwtr}
Let $n$ be a prime number, then the centralizer of $A\in K_T\{X\}$ is integrally closed in $K_T\{X\}$.
\end{prop}

\begin{proof}
  Due to the results of the next section \ref{GenMatrIntClosed}, the  algebra of generic matrices with trace (with forms in positive characteristic case)  is integrally closed, and therefore it is only necessary to prove that the rational closure of the centralizer is within the algebra of generic matrices with trace. But in a domain, the fraction of two elements commuting with $x$ also commutes with $x$ and hence we are done.
\end{proof}

\subsection{Algebra of generic matrices with traces (forms) is integrally closed}   \label{GenMatrIntClosed}

Our main goal in this subsection is to prove Theorem \ref{ThMain2}.  Recall the following result:

\medskip

{\bf Theorem} {\it
  Let $K$ be a field of characteristic zero. Then the algebra of generic matrices with trace over $K$ is integrally closed and moreover, if $\Ch(K)>0$ then the algebra of generic matrices with forms (i.e. characteristic coefficients) over $K$ is integrally closed.}

\medskip

\subsubsection{Realization of the algebra of generic matrices as a generic algebra over a skew field}

Let $K$ be an algebraically closed field.  Consider the ring $K\{X\}$ of $n\times n$ generic matrices over $K$,
with $s$ number of generators. Let $Z$  be its center and $F$ the field of fractions of $Z$. Then $Q=F\otimes K\{X\}$ is a skew field, $n^2$-dimensional over $F$ and it is $PI$-equivalent to $K\{X\}$. One can naturally define traces (forms) on $Q$.

Now we consider \emph{the algebra of generic elements} of $Q$. It is isomorphic to the algebra of generic matrices over $F$. Naturally, it is also possible to define the algebra of generic elements of $Q$ with traces (forms).

The construction of the algebra of generic elements proceeds analogously to the construction of generic matrices. Let
$e_1,\dots,e_{n^2}$ be an $F$-basis of $Q$. Consider the set of variables $x_{i}^{(k)}$, $i=1,\dots,n^2$, $k=1,\dots,s$. Let
$$A_k=\sum_{i=1}^{n^2} e_ix_{i}^{(k)},$$
and consider $A=F[A_1,\dots,A_s]$; this is \emph{the algebra of generic elements} over $F$. One can define an algebra of generic elements of any finitely dimensional algebra $B$ (not necessary associative) over an infinite field\footnote{Constructions of generic algebras over a finite field is discussed in \cite{BVR}.}. Such algebra is always relatively free. Two generic algebras for $B_1$ and $B_2$ are isomorphic if and only if $B_1$ and $B_2$ are $\PI$-equivalent. The algebra of generic matrices of order $n$ with $s$ generators over $F$ is therefore naturally isomorphic to our generic elements algebra of a skew field $Q$.

\subsubsection{Algebra of generic elements of a skew field is integrally closed}

We seek to establish the integral closure property of the algebra of generic elements of a skew field in this subsection.

Consider the algebra $Q\otimes F(y_1,\dots,y_m)$, which is a skew field. Any element $r$ of this skew field can be represented as
$$r=\sum_{i=1}^{n^2}e_iP_i/Q_i,$$
where $P_i$, $Q_i$ do not have common divisors. Let $R$ be the least common multiple of $Q_i$. We call $R$ a {\it denominator} of $r$. We continue with some observations:

\begin{lem}    \label{LeQotimesDomain}
\begin{enumerate}[label=\alph*)]
    \item Let $D$ be a commutative domain over $F$. Then $Q\otimes_FD$ is a non-commutative domain, its skew field localization is isomorphic to $Q\otimes_F{\cal D}$, where $\cal D$ is a field of fractions of $D$.
    \item Let $D$ be a commutative domain over $F$ and $I\triangleleft D$ a prime ideal. Then $Q\otimes_FD/(I\cdot Q\otimes_FD)\simeq Q\otimes_FD/I$ is a domain.
\end{enumerate}
\end{lem}

In what follows, we will work with generic elements of the skew field $Q$ instead of generic matrices.

\begin{lem}\label{denominator}
If $R$ is a denominator of $r$, then $R^n$ is a denominator of $r^n$.
\end{lem}

\begin{proof}
It is enough to treat the case where $R$ is a power of an irreducible polynomial. Consider the element $R\cdot r$. It
belongs to $Q\otimes F[y_1,\dots,y_m]$, it is not divisible by $\bar{R}$, the irreducible component of $R$ (because otherwise $R$ is not a minimal denominator of $r$).

Consider the integral domain $F[y_1,\dots,y_m]/\langle\bar{R}\rangle$ (see Lemma \ref{LeQotimesDomain}). Suppose that its localization is a skew field $G$. Now consider $ Q\otimes F[y_1,\dots,y_m]/\langle\bar{R}\rangle \subset Q\otimes G$. Because $Q\otimes G$ is a skew field, hence it does not contain any nilpotent element and therefore $(R\cdot r)^n$ is not divisible by $\bar{R}$ and the statement of the lemma follows.
\end{proof}

\begin{coro}    \label{CoIntClosQpolyn}
The algebra $Q\otimes_F F[z_1,\dots,z_m]$ of polynomials over $Q$ is integrally closed.
\end{coro}

\begin{proof}
  If $A$ belongs to the algebra of proper rational functions over $Q$, then it has a non-trivial denominator $R$. Then by the previous lemma, $A^n$ has the denominator $R^n$. Every power of $A$ also has a non-trivial denominator and can not be a polynomial over $Q$. 
  Moreover, consider $P(A)=A^n+S(A) (n>1)$ be a polynomial of $A$ with $\deg(S) = k< n$. Denominator of any component of $s(A)$ is $R^k$, with $k< n$. Hence, $P(A)$  has the same non-trivial denominator $R^n$ as $A^n$ by Lemma \ref{denominator}. Then $P(A)$ can not be a polynomial over $Q$.
\end{proof}

\begin{remark} In general, the same statement as in Corollary \ref{CoIntClosQpolyn} is not true for matrices. Consider the partial case when $P(A)=A^n$, because if $A^n$ (the $n^{th}$ power of a matrix $A$ with $\det(A)\ne 0$ over rational functions) is a matrix over the polynomial ring, then it does not imply that $A$ is a matrix with polynomial entries. That is why we represent the algebra of generic matrices generated by generic elements of a skew field.
\end{remark}

Let $P(A)$ be a non-constant polynomial of $A$, we proved that if $P(A)$ is a polynomial matrix over $Q$ and $A\in Q\otimes F(y_1,\dots,y_m)$ then $A\in Q\otimes F[y_1,\dots,y_m]$ and this was the reason to consider ``generic $Q$-elements".

\subsubsection{Trace algebras}

In this subsection, we will use traces. A trace algebra is a generalized monoid (a set with a ternary operation that satisfies certain generalized associativity and identity laws). Every trace algebra induces in a very natural way a mathematical object which exhibits the behavior of the interrelations familiar from the theory of linear spaces (with the notable exception of ``Fubini's theorem''). The induced object is induced in a well-behaved manner: its structure is determined by the structure of the trace algebra, and by nothing else. Conversely, if the trace algebra is well-behaved, then it is uniquely determined by its induced object. This means that when everything is well behaved, then our abstract ``linear spaces'' are the same thing as trace algebras. For more information regarding the trace algebras, we refer the interested reader to \cite{S72}.

\begin{remark}
Note that $\tr(e_i)\in F$ and $\tr(\sum x_ie_i)=\sum x_i\tr(e_i)$.
\end{remark}

And we have the following results.

\begin{lem}
Let $X=e_1x_1+e_2x_2+\cdots+e_{n^2}x_{n^2}$ be a generic $Q$-element, suppose $\{y_1,\dots,y_m\}$ is a set of variables disjoint with $x_1,\dots,x_{n^2}$. Let $S\in Q(y_1,\dots,y_m)$, $\tr(XS)\in Q[x_1,\dots,x_{n^2},y_1,\dots,y_m]$.  Then $S\in Q[y_1,\dots,y_m]$.
\end{lem}

\begin{proof}
  We have to prove that if $S$ has a non-trivial denominator, then $\tr(XS)$ also has a non-trivial denominator. Taking the tensor product with the algebraic closure $F^{cl}$ of $F$ we come to the matrix algebra $M_n(F^{cl})$ instead of $Q$. Because $\tr(X_1SX_2)=\tr(X_2X_1S)$ and we may consider an expression  $\tr(X_1SX_2)$ instead of $\tr(XS)$. It is enough to consider one proper specialization of $X_1,X_2$ such that $\tr(X_2SX_1)$ has a non-trivial denominator.

  Indeed, if $S$ has a non-trivial denominator, then one of its entries, $S_{ij}$, also has a nontrivial denominator. Then we can specialize $X_1$ to $E_{1i}$ and $X_2$ to $E_{j1}$. Then $\tr(X_2SX_1)=S_{ij}$ and will have the same denominator, which is not trivial.
\end{proof}

\begin{lem} \label{LeTrMtrTrans}
Let $X_s=\sum_{i=1}^{n^2}e_iy_i$ be a generic matrix. Let $B\in Q\otimes F[x_{i}^{(k)}]$ for $k=1,\dots,s-1$ and let $\{x_i\}$, $\{y_j\}$ be sets of variables with empty intersection. If $\tr(XB)$ is the trace of a generic matrix with trace, then $B$ is a generic matrix with trace.
\end{lem}

\begin{proof} In order to prove this result we use the standard method from $PI$-theory.

  Note that if variables participating in a generic matrix $X$ are different from $y_1,\dots,y_m$ then for $B\in M_n(k[y_1,\dots,y_m])$ the mapping $B\to \tr(XB)$ is injective and hence if $\tr(XB)=\tr(XB_1)$ for some generic matrix $B_1\in M_n(k[y_1,\dots,y_m])$ then $B=B_1$.

  Let $\tr(XB)$ be a trace polynomial, it has degree 1 with respect to $X$. It is an element of the algebra of generic matrices with trace and it is a sum of monomials of the form $$\tr(XB_1B_2\cdots B_l)\tr(M_1)\tr(M_2)\cdots\tr(M_l).$$
   Now if  we replace each such monomial by
  $B_1B_2\cdots B_l\tr(M_1)\tr(M_2)\cdots\tr(M_l)$ we will get $\sum B_1B_2\cdots B_l\tr(M_1)\tr(M_2)\cdots\tr(M_l)=B$. Therefore $B$ belongs to the algebra of generic matrices with trace and hence we are done.
\end{proof}

\begin{remark}
  We need to emphasize that the constructions related to trace monomials and the special variable $X$ are rather common in $\PI$-theory (see \cites{Kemer,Belovshpecht,Razmyslov} for example).
\end{remark}

\subsubsection{The proof}

Now we are ready to prove the main result of this section on the integral closedness of the algebra of generic matrices with traces (for the positive characteristic case -- with characteristic coefficients).

We first embed the algebra of generic matrices with traces into $Q[y_1,\dots,y_m]$ (for $y_1,\dots,y_m$ the entries coefficients) which is integrally closed (corollary \ref{CoIntClosQpolyn}) and hence our integral closure lies in $Q[y_1,\dots,y_m]$.

Let $Y$ be in the integral closure. Let $X$ be a new generic matrix variable, not participating in $Y^m$. Then
$\tr(XY)$ belongs to the localization of traces and therefore to the field of invariants. On the other hand, it belongs to
$F[x_1,\dots,x_{n^2},y_1,\dots,y_m]$.

Because of Procesi's theorem (resp. Donkin's theorem in positive characteristics), it belongs to the trace algebra
(resp. characteristic coefficients generic algebra for positive characteristics case).  Therefore $\tr(XY)$ is a characteristic
coefficients polynomial.

According to Lemma \ref{LeTrMtrTrans}, an element of the integral closure $Y$ belongs to the algebra of generic matrices with trace (forms or characteristic coefficients in positive characteristics), which is what we need.

\subsection{Proof of integrally closedness of centralizers}

In this subsection, we will prove that the centralizer $C$ is integrally closed. First, we will try to discuss the problem in the general case, and then we will proceed with the proof.

\medskip

As before, let $K$ be our ground field. Denote by $K\llangle X\rrangle $ the $K$-algebra of formal power series with
$X=\{x_1,\dots,x_n\}$. P. Cohn\cite{Cohn} proved that if $f\in K\llangle X\rrangle $ is not a constant,  then
$C(f;K\llangle X\rrangle)=K[\![\ell]\!]$, for some power series  $\ell$, where $K[\![\ell]\!]$ stands for the ring of
formal power series in $\ell$. This result is known as \emph{Cohn's centralizer theorem}.
\medskip

By Cohn's centralizer theorem, the centralizer of every non-constant element in $K\llangle X\rrangle$ is commutative, and since $K\langle X\rangle $ is a $K$-subalgebra of $K\llangle X\rrangle$, the centralizer of a non-constant element of $K\langle X\rangle $ will be commutative as well.

\medskip

We will give proof of the following Theorem by using the generic matrices technique based on our previous work \cite{KBRZh}.

\begin{thm}\label{integral}
The centralizer $C$ of  a non-trivial element $f$ in the free associative algebra is integrally closed.
\end{thm}

Let $g,P,Q\in C:=C(f;K\langle X\rangle)$, and suppose $g=P/Q$, then there exists $h\in C$, such that $R(h)=g$. This means that the centralizer $C$ is integrally closed. Next, we will establish a relationship between the centralizer and the algebra of generic matrices $K\{X\}$ by the local isomorphism.

Consider the homomorphism $\pi$ from the free associative algebra $K\langle X\rangle$ to $K_T\{X\}$, the algebra of generic matrices with traces. Let us denote by $\bar{g}$ the image $\pi(g)$. Then we have the following proposition.

\begin{prop}
Consider the epimorphism $\pi: K\langle X\rangle\to K_T\{X\}$. Let the order of matrices be a prime number $p\gg 0$. Consider $\overline{g}=\pi(g), \overline{P}=\pi(P)$ and $\overline{Q}=\pi(Q)$. Then there exists $\overline{h}\in K_T\{X\}$ such that we have the following assertions:
\begin{enumerate}
\item[1)] $R(\bar{h})=\overline{g}$, \label{Equ:1)}
\item[2)] $\bar{h}=\frac{\overline{P}}{\overline{Q}}$, \label{Equ:2)}
\item[3)] $\bar{h}\in \overline{C}$, where $\overline{C}=\pi(C)$. \label{Equ:3)}
\end{enumerate}
\end{prop}

\begin{proof}
Statements 1) and 2) will directly follow from Proposition \ref{ICwtr} which indicates that the algebra of generic matrices with traces is integrally closed. Hence, we just need to prove the third statement. Note that all eigenvalues of $\bar{g}$ are pairwise distinct due to Proposition \ref{propeigen}. The same also holds for $\bar{f}$. Therefore $\bar{f}$ and,$\bar{g}$ are diagonalizable and $\bar{h}$ can be diagonalized in the same eigenvectors basis and hence by Proposition \ref{ProPcomm} it follows that $\bar{h}$ commutes with $\bar{f}$, meaning that $\bar{h}\in \overline{C}$.
\end{proof}

Now we need to prove that $\bar{h}$ in fact belongs to the algebra of generic matrices without a trace. We use the local isomorphism to dispose of traces. One may refer to \cite{BP87} for more details concerning the local isomorphism.

\medskip

\begin{defn}[Local isomorphism]
Let $\mathbb{A}$ be an algebra with generators $a_1,\dots,a_s$ and homogeneous with respect to this set of generators, and let $\mathbb{A}'$ be an algebra with generators $a_1',\dots,a_s'$ homogeneous with respect to this set of generators. We say that $\mathbb{A}$ and $\mathbb{A}'$ are {\em locally $L$-isomorphic} if there exists a linear map $\varphi:a_i\to a_i'$ on the space of monomials of degree $\leq 2L$, and in this case for any two
elements $b_1,b_2\in A$ with highest term of degree $\leq L$, we have $$b_i=\sum_{j}M_{ij}(a_1,\dots,a_s)~ \text{and}~~
b_i'=\sum_{j}M_{ij}(a_1',\dots,a_s'),$$ where $M_{ij}$ are monomials, and for $b=b_1\cdot b_2,~ b'=b_1'\cdot b_2'$, we have $\varphi(b)=b'$.
\end{defn}

For to continue we need the following lemmas and propositions:

\begin{lem}[Local isomorphism lemma]\label{lil}
For any $L$, if $s$ is a large enough prime, then the algebra of generic upper triangular matrices $\mathbb{U}_s$ is
locally $L$-isomorphic to the free associative algebra. Also the reduction of the algebra of generic matrices with traces of degree $n$ provides an isomorphism up to degree $\leq 2s$.
\end{lem}

It is important to note a well-known and extremely helpful fact here, which we will put as a proposition.

\begin{prop}
The trace of every element in $\mathbb{U}_s$ of any characteristic is zero.
\end{prop}

We also substantiated the following results in the aforementioned procedure:

\begin{prop}
If $n>n(L)$, then the algebra of generic matrices with and without traces are locally isomorphic. Hence, the algebra of generic matrices without traces is $L$-locally integrally closed.
\end{prop}

\begin{lem}
Consider the projection $\overline{\pi}$ from the algebra of generic matrices with traces to $\mathbb{U}_s$, sending all traces to zero. Then we have 
$$\overline{\pi}(R(\overline{h}))=\overline{\pi}(g).$$
\end{lem}
 
Now we are ready to prove Theorem \ref{integral}, which is our main result in this section:
\begin{proof}[Proof of Theorem \ref{integral}]
Let $p$ be a big enough prime number and because space $K_T\{X\}$ of degree $\leq p$, is isomorphic to the space of free associative algebras, hence can set 
$$p\geq  2(\deg(f)+\deg(g)+\deg(P)+\deg(Q)).$$ 
 Note that in this case up to isomorphism, we have that $h$ corresponds to $\bar{h}$ and due to local isomorphism, $R(h)=g$ and $h=P/Q$, meaning that $hQ=P$.  Also, we have $h$ commutes with $f$, which means that $h\in C$.
\end{proof}

\section{Proof of the Bergman centralizer theorem} \label{final-proof}

We can infer the following proposition from the previous two sections:

\begin{prop}\label{propratint}
Let $p$ be a large enough prime number, and $K\{X\}$  the algebra of generic matrices of order $p$. For any $A\in K\{X\}$, the centralizer of $A$ is integrally closed in $K\{X\}$ over the center of $K\{X\}$.
\end{prop}

In our previous paper \cite{KBRZh}, we established that the centralizer in the algebra of generic matrices is a commutative ring of transcendence degree one. According to Proposition \ref{propratint}, $C(A)$ is integrally closed in $K\{X\}$ and if $p$ is big enough, then $K\{X\}$ is $L$-locally integrally closed.

\medskip

Let $X$ be a totally ordered set and $W$ be the free monoid on the set $X$ with the empty word representing $1$. Let $\overline{W}$ be the set of all \textit{right infinite} words in $X$ (i.e. infinite sequences of elements of $X$). Given $u\in  W-\{1\}$, let $u^{\infty}$ denote by the word obtained by repeating $u$ infinitely: $uuu\dots$. Let $\overline{W}$ be ordered lexicographically.

Next we need the following extra fact about ``infinite words" from Bergman \cite{Berg}. 

\begin{lem}[Bergman]\label{lemword}
Let $u,v\in W\setminus\{1\}.$ If $u^{\infty}>v^{\infty}$, then we have $u^{\infty}>(uv)^{\infty}>(vu)^{\infty}>v^{\infty}$.
\end{lem}
\begin{proof}
It suffices to show that the whole inequality is implied by $(uv)^{\infty}>(vu)^{\infty}$. Suppose $(uv)^{\infty}>(vu)^{\infty}$, then we have the following
$$(vu)^{\infty}=v(uv)^{\infty}>v(vu)^{\infty}=v^2(uv)^{\infty}>v^2(vu)^{\infty}=\cdots =v^{\infty}.$$
Similarly, we can obtain $(uv)^{\infty}<u^{\infty}.$
\end{proof}

It is easy to see that Lemma \ref{lemword} still works when we replace ``$>$'' with ``$=$'' or ``$<$''.

\medskip

\begin{remark}
Similar constructions are used in \cite{BBL} for Burnside type problems and the Shirshov height theorem.
\end{remark}

\medskip

Now let $R$ be the semigroup algebra on $W$ over field $K$, meaning that $R=K\langle X\rangle$ is the free associative algebra. Consider $z\in \overline{W}$ be an infinite period word, and we denote by $R_{(z)}$ the $K$-subspace of
$R$ generated by words $u$ such that $u=1$ or $u^{\infty}\leq z$. Let $I_{(z)}$ be the $K$-subspace spanned by words $u$ such that $u\neq 1$ and $u^{\infty}<z$. Using Lemma \ref{lemword}, we can conclude that $R_{(z)}$ is a subring of $R$ and $I_{{z}}$ is a two-sided ideal in $R_{(z)}$ and it follows that $R_{(z)}/I_{(z)}$ is isomorphic  to the polynomial ring $K[u]$.

\begin{prop}[Bergman]\label{propberg}
For $C\neq K$ (a finitely generated subalgebra of $K\langle X\rangle$) there is a homomorphism $f$ of $C$ into the polynomial algebra over $K$ in one variable, such that $f(C)\neq K$.
\end{prop}

\begin{proof}[Proof (Bergman)]
For a given totally order on $X$, let $G$ be a finite set of generators for $C$ and let $z$ be the maximum over all monomials $u\neq 1$ with nonzero coefficient in elements of $G$ of $u^{\infty}$. Then we have $G\subseteq R_{(z)}$ and hence $C\subseteq R_{(z)}$, and the quotient map $f: R_{(z)}\to R_{(z)}/I_{(z)}\cong K[v]$ is nontrivial on $C$.
\end{proof}

\begin{remark}
    Because $f(C)\neq K$ and $C$ does not have zero-divisors and is a commutative ring of transcendental degree one \cite{KBRZh}, $f$ is a monomorphism.
\end{remark}

Now we are ready to \textbf{finish the proof of Bergman's centralizer theorem}.

\medskip

Consider the homomorphism from Proposition \ref{propberg}. Because $C$ is a centralizer of an element in $K\langle X\rangle\setminus K$, it has transcendence degree 1. Consider the homomorphism $\rho$ which sends $C$ to the ring of polynomials. The homomorphism has kernel zero, otherwise $\rho(C)$ will have a smaller transcendence degree.
Note that $C$ is integrally closed and finitely generated, therefore it can be embedded into the polynomial ring in one indeterminate. Since $C$ is integrally closed, it is isomorphic to the polynomial ring in one indeterminate.

\medskip

Consider the set of system of $C_{\ell}$, finitely generated subalgebras $\ell$ generators of the centralizer algebra $C$, then $C=\cup_{\ell}C_{\ell}$. Let $\overline{C_{\ell}}$ be the integral closure of $C_{\ell}$. Because $C_{\ell}$ is a one dimensional domain, which can be embedded to the ring of polynomial with one variable, then $\overline{C_{\ell}}$ the integral closure of those images is isomorphic to the ring of polynomial, i.e. $\overline{C_{\ell}}=K[z_{\ell}]$, where $z_{\ell}$ belongs to the integral closure of $C_{\ell}$. Consider sequence of $z_{\ell}$. We have $K[z_{\ell}]\subseteq K[z_{\ell+1}] \subseteq \cdots$. If $K[z_{\ell}] \subsetneq K[z_{\ell+1}]$, then the degree of $z_{\ell}$ is strictly less than the degree of $z_{\ell+1}$, otherwise $K[z_{\ell}]= K[z_{\ell+1}]$ and we can put $z_{\ell}=z_{\ell+1}$.
Hence this sequence stabilizes for some element $x$ and it shows that $K[z]$ is the needed centralizer.

\section{On the rationality of degree one subfields in the skew field of fractions of generic matrices}    \label{Discourtions}

The following open problem will be discussed in this final section along with some potential solutions. Unless otherwise stated, we still use the same notations from the previous section.

\medskip

\begin{problem}
Consider the algebra of generic matrices $K\{X\}$ of order $s$. Consider its	skew field of fractions, $\Frac(K\{X\})$ \cite{Kol00}, and let $\mathcal{K}$ be a subfield of $\Frac(K\{X\})$ of transcendence degree one over the base field $K$. Now the question is whether we can say $\mathcal{K}$ is isomorphic to the field of rational functions over $K$. In other words, do we have $\mathcal{K}\cong K(t)$?
\end{problem}

\medskip

Let $K\{X\}$ be the algebra of generic matrices of a large enough prime order $s:=p$. Let $\Lambda$ be the diagonal generic matrix $\Lambda=\diag{(\lambda_1,\dots,\lambda_s)}$ in $K\{X\}$, where the transcendence degree satisfies $\Trdeg K[\lambda_i]=1$. Now let $N$ be another generic matrix, whose coefficients are algebraically independent from $\lambda_1,\dots,\lambda_s$. This means that if $R$ is a ring of all coefficients of $N$, with
$\Trdeg(R)=s^2$, then 
$$\Trdeg R[\lambda_1,\dots,\lambda_s]=s^2+\Trdeg{K[\lambda_1,\dots,\lambda_s]}.$$

\begin{prop}\label{ideal}
Let us consider the generic matrices $f$ and $g$.
\begin{enumerate}
\item[a)] Let $K[f_{ij}]$ be a commutative ring, and $I=\langle f_{1i}\rangle\lhd K[f_{ij}]$, for $(i>1)$, be an ideal of $K[f_{ij}]$. Then we have $K[f_{11}]\cap I=0$.
\item[b)] Let $K[f_{ij},g_{ij}]$ be a commutative ring, and $J=\langle f_{1j},g_{1j}\rangle\lhd K[f_{ij},g_{ij}]$ when $(i,j>1)$. If $f$ and $g$ are algebraic dependent on $e_1$, namely there exists a polynomial $P$ with $P(f_{11},g_{11})=0$, then $K[f_{11},g_{11}]\cap J=0$.
\end{enumerate}
\end{prop}

\begin{coro}\label{coro32}
Let $\mathbb{A}$ be an algebra of generic matrices generated by $a_1,\dots,a_{s},a_{s+1}$. Let $f\in K[a_1,\dots,a_s]$ and $\varphi=a_{s+1}fa_{s+1}^{-1}$. Let $I=\langle \varphi_{1i}\rangle\lhd K[a_{1},\dots,a_{s+1}]$. Then $K[\varphi_{11}]\cap I=0$.
\end{coro}

\begin{proof}
Note that we have $f=\tau \Lambda \tau^{-1}$ for some $\tau$ and a diagonal matrix $\Lambda$ by Proposition \ref{ideal}. Then $\varphi=(a_{s+1}\tau)\Lambda(a_{s+1}\tau)^{-1}$ and we can treat $(a_{s+1}\tau)$ as a generic matrix.
\end{proof}

\begin{thm}\label{thmover}
Let $C:=C(f;K\langle X\rangle)$ be the centralizer ring of $f\in K\langle X\rangle\setminus K$ and let $\overline{C}$ be the reduction of generic matrices, and $\overline{\overline{C}}$ be the reduction with respect to the first eigenvalue action. Then we have $\overline{\overline{C}}\cong C$.
\end{thm}

\begin{proof}
Recall that we already have shown $\overline{C}\cong C$  satisfies \cite{KBRZh}. If we have $P(g_1,g_2)=0$, then clearly
$P(\lambda_1(g_1),\lambda_2(g_2))=0$ in the reduction on the first eigenvalue action. Suppose $\overline{\overline{P(g_1,g_2)}}=0$. Then $P(g_1,g_2)$ is an element of the ring of generic matrices with at least one zero eigenvalues and because the minimal polynomial is irreducible, and hence implies that $P(g_1,g_2)=0$. This means that any reduction with $\lambda_1$ satisfies the equation completely, and it is exactly what we were looking for!
\end{proof}

\medskip

Consider $\overline{C}$ as before and for any $\overline{g}=(g_{ij})\in\overline{C}$, let us investigate $g_{11}$. As before suppose there is a polynomial $P$ with coefficients in $K$, such that $P(f,g)=0$. By Proposition \ref{ideal}, we make the intersection of ideals $J=\langle f_{1j},g_{1j}\rangle$ for $(j>1)$ and even stronger, meaning that $K[f_{11},g_{11}]\cap J=0$ and from Theorem \ref{thmover} we can retrieve the following proposition.

\begin{prop}
$$K[f_{11},g_{11}] \cong K[f,g] \mod J.$$
\end{prop}

\begin{proof}
Consider matrices $f$, $g$ in $\mod J$ respectively with $\overline{\overline{f}}$ and $\overline{\overline{g}}$ as follows:
$$\overline{\overline{f}}=\begin{pmatrix}
   \lambda_1 & 0 & \cdots & 0 \\
   * & * & * & * \\
   * & * & * & * \\
   * & * & * & *
  \end{pmatrix}, ~~
  \overline{\overline{g}}=\begin{pmatrix}
   \lambda_2 & 0 & \cdots & 0 \\
   * & * & * & * \\
   * & * & * & * \\
   * & * & * & *
  \end{pmatrix}.$$
Then for any $P(\overline{\overline{f}},\overline{\overline{g}}) \mod J$, we have
$$\overline{\overline{f}}=\begin{pmatrix}
   P(\lambda_1,\lambda_2) & 0 & \cdots & 0 \\
   * & * & * & * \\
   * & * & * & * \\
   * & * & * & *
  \end{pmatrix}$$
\end{proof}

Now, let us propose another approach as follows. Consider $K(f,g)$ as before and let us extend the algebra of generic matrices by a new matrix $T$, independent from the previous ones. Consider conjugation of $P(f,g)$ with $T$ by $T P(f,g) T^{-1}$, and let $\tilde{f}=TfT^{-1}$ and $\tilde{g}=TgT^{-1}$. By Corollary \ref{coro32}, we have 
$$K[f_{11},g_{11}]\cap J=0,$$ 
which means that 
$$P(f_{11},g_{11})=0 \mod J,$$ 
for $f_{11}$ and $g_{11}$ polynomials over the commutative ring generated by all entries of $K[f,g]$ and $T$. Then $\Frac (K(f,g))$ can be embedded into the fractional field of the ring of polynomials and according to L\"{u}roth's Theorem, $\Frac(K(f,g))$ (hence $\Frac(C)$) will be isomorphic to the field of rational functions in one variable.

Note well that this will not guarantee the rationality of our field, and there are counterexamples in this situation. However, this approach seems to be useful for the leading term analysis.

\section*{Data availability statement}
Our manuscript has no associate data.

\end{document}